\documentclass{amsart}

\usepackage{amssymb,amsmath}
\usepackage[section]{algorithm}
\usepackage{algpseudocode}

\newtheorem{theorem}{Theorem}[section]
\newtheorem{proposition}[theorem]{Proposition}

\newtheorem{corollary}[theorem]{Corollary}

\newtheorem{definition}[theorem]{Definition}

\newtheorem{remark}[theorem]{Remark}

\newtheorem{example}[theorem]{Example}

\def\ZZ{{\mathbb Z}}
\def\QQ{{\mathbb Q}}
\def\KK{{\mathbb K}}

\def\Ker{{\mathrm{Ker\,}}}
\def\Im{{\mathrm{Im\,}}}

\def\Mon{{\mathrm{Mon}}}

\def\Gr{Gr\"obner}

\def\lm{{\mathrm{lm}}}
\def\lc{{\mathrm{lc}}}

\def\LM{{\mathrm{LM}}}

\def\HF{{\mathrm{HF}}}
\def\HS{{\mathrm{HS}}}

\def\mi{{\mathbf{mi}}}
\def\ri{{\mathbf{ri}}}

\def\OrbitData{{\textsc{OrbitData}}}

\def\F{{\mathcal{F}}}
\def\O{{\mathcal{O}}}
\def\N{{\mathcal{N}}}

\def\B{{\mathbf{B}}}
\def\C{{\mathbf{C}}}
\def\vH{{\mathbf{H}}}

\begin{document}

\title[Monomial right ideals and Hilbert series $\ldots$]
{Monomial right ideals and the Hilbert series of noncommutative modules}

\author[R. La Scala]{Roberto La Scala$^*$}

\address{$^*$ Dipartimento di Matematica, Universit\`a di Bari, Via Orabona 4,
70125 Bari, Italy}
\email{roberto.lascala@uniba.it}

\thanks{Partially supported by Universit\`a di Bari}



\maketitle

\begin{quotation}
{\em Per Armando, in memoriam.}
\end{quotation}

\begin{abstract}
In this paper we present a procedure for computing the rational sum
of the Hilbert series of a finitely generated monomial right module $N$
over the free associative algebra $\KK\langle x_1,\ldots,x_n \rangle$.
We show that such procedure terminates, that is, the rational sum exists,
when all the cyclic submodules decomposing $N$ are annihilated
by monomial right ideals whose monomials define regular formal languages.
The method is based on the iterative application of the colon right
ideal operation to monomial ideals which are given by an eventual
infinite basis. By using automata theory, we prove that the number
of these iterations is a minimal one. In fact, we have experimented
efficient computations with an implementation of the procedure in Maple
which is the first general one for noncommutative Hilbert series.
\end{abstract}

%% VERSION
%\begin{center} MAR 2016 \end{center}

%%%%%%%%%%%%%%%%%%%%%%%%%%%%%%%%%%%%%%%%%%%%%%%%%%%%%%%%%%%%%%%%%%%%%%

\section{Introduction}

It is difficult to list the plenty of objects that are represented
as a finite sequence of symbols, that is, by a word or a string. Let us
think for example to our own dna, a message in a code, a word in a natural
or formal language, the base expansion of a number, a path in a graph or
a transition of states in a machine, etc. From the perspective of the
algebraist these elements are just monomials of the free associative algebra
$F = \KK\langle x_1,\ldots, x_n \rangle$, that is, the algebra of polynomials
in the noncommutative variables $x_i$. In fact, any finitely generated
(associative) algebra is isomorphic to a quotient $F/I$ where $I\subset F$
is the two-sided ideal of the relations satisfied by the generators of the
algebra. Another important generalization is the notion of finitely generated
right $F$-module which corresponds to a quotient $F^r/M$ of the free right
$F$-module $F^r = F\oplus \cdots \oplus F$ with respect to some submodule
$M\subset F^r$. It is probably useless to mention the importance of the
noncommutative structures, for instance, in physics but let us just refer
to the recent book \cite{Su}.

Fundamental data about algebras and modules consist in their dimension
over the base field $\KK$ or in their ``growth'' in case they
are infinite dimensional. By definition, the growth or the affine Hilbert
function of $N = F^r/M$ is the map $d\mapsto \dim F^r_{\leq d}/M_{\leq d}$
where $F_{\leq d}\subset F$ is the subspace of the polynomials of degree
at most $d$ and $M_{\leq d} = M\cap F^r_{\leq d}$. This sequence is naturally
encoded by its generating function which is called the affine Hilbert
series of $N$.

Unlike the commutative case, owing to non-Noetherianity of the free associative
algebra one has that this series has generally not a rational sum.
Nevertheless, finitely generated free right modules have clearly a rational
Hilbert series and the same happens in fact for all finitely presented modules.
This is because $F$ is a free ideal ring \cite{Co} which implies that
any right submodule $M\subset F^r$ is in fact free.
It is then natural to ask if there are other cases when the Hilbert series
is rational. One possible approach consists in reducing the problem to
the monomial cyclic case that is to right modules $C = F/I$ where $I\subset F$
is a monomial right ideal. This can be done by using a suitable monomial
ordering of $F^r$ and by proving a noncommutative analogue of a well-known
theorem of Macaulay. We present these results in Section 2.

A short exact sequence for a monomial cyclic right module $C = F/I$ is presented
in Section 3 allowing the possibility to reduce the computation of the Hilbert
series of $C$ to the one of $n$ cyclic modules defined by the colon right
ideals of $I$ with respect to each variable $x_i$.
By iterating this exact sequence one obtains in principle a rational sum
for the Hilbert series of $C$ but the main problem is to provide that
one has only a finite number of iterations. Under such assumption,
in Section 4 we present all details of the corresponding algorithm
which is based on the explicit description of a monomial basis for the
colon right ideals of a monomial ideal with respect to a monomial of $F$.

In Section 5 we explain that the set of the colon right ideals of $I$
which are involved in the computation of the Hilbert series of the
monomial cyclic right module $C = F/I$ are in fact the states
of a minimal deterministic automaton recognizing the formal language
given by the monomials of $I$. By the Kleene's theorem this implies that
the proposed procedure has termination if and only if such language is
a regular one, that is, all the monomials of $I$ can be obtained from
finite subsets by applying a finite number of elementary operations.
For the two-sided case, one finds in \cite{Ma,Uf2} the idea that
the rationality of the Hilbert series of an algebra $A = F/I$ is provided
by the existence of a deterministic finite automaton recognizing the normal
monomials modulo $I$ or, equivalently, the monomials of $I$.
In this paper we present a complete and effective direct method
for costructing a minimal such automaton. The roots of our approach can
be traced back to the formal languages notion of Nerode equivalence that
we have enhanced for right and two-sided monomial ideals by means of
the tools of algebra.

The proposed ideas results in a feasible algorithm which has been illustrated
in Section 6 for an infinitely related monomial algebra that we obtain from
an Artin group. In Section 7, by means of an experimental implementation that
we have developed with Maple, we present some data and timings for the
computation of the Hilbert series of a couple of Hecke algebras. Finally,
some conclusions and ideas for further research directions are proposed
in Section 8.

%%%%%%%%%%%%%%%%%%%%%%%%%%%%%%%%%%%%%%%%%%%%%%%%%%%%%%%%%%%%%%%%%%%%%%

\section{Hilbert functions and series}

Let $\KK$ be any field and let $X = \{x_1,\ldots,x_n\}$ be any finite set.
We denote by $F = \KK\langle X\rangle$ the free associative algebra which is
freely generated by $X$. In other words, the elements of $F$ are noncommutative
polynomials in the variables $x_i$. We define then $W = \Mon(F)$
the set of all monomials of $F$, that is, the elements of $W$ are words
over the alphabet $X$. We consider $F$ as a graded algebra by means of
the standard grading which defines $\deg(w)$ as the length of a word $w\in W$.
For any integer $r > 0$, denote by $F^r$ the (finitely generated) free right
$F$-module of rank $r$. If $\{e_i\}$ is the canonical basis of $F^r$ then
the elements of such module are the right $F$-linear combinations
$\sum_i e_i f_i$ ($f_i\in F$). We endow the module $F^r$ with the standard
grading, that is, we put $\deg(e_i w) = \deg(w)$, for any $1\leq i\leq r$
and for each $w\in W$. This implies that $F^r$ is also a filtered module
by means of the canonical filtration $F^r = \bigcup_{d\geq 0} F^r_{\leq d}$
where $F^r_{\leq d} = \sum_{0\leq k\leq d} F^r_k$, for any $d\geq 0$.
We recall now the notion of Hilbert function and series for finitely
generated right $F$-modules. 

\begin{definition}
Let $M$ be a right submodule of $F^r$ and define the quotient right module
$N = F^r/M$. For all $d\geq 0$, we define $\HF_a(N)(d) =
\dim_\KK F^r_{\leq d}/M_{\leq d}$ where $M_{\leq d} = M\cap F^r_{\leq d}$.
Moreover, we put
\[
\HF(N)(d) =
\left\{
\begin{array}{cl}
\HF_a(N)(0) & \mbox{if}\ d = 0; \\
\HF_a(N)(d) - \HF_a(N)(d-1) & \mbox{otherwise}.
\end{array}
\right.
\]
We call the sequences $\HF_a(N) = \{\HF_a(N)(d)\}_{d\geq 0}$ and $\HF(N) =
\{\HF(N)(d)\}_{d\geq 0}$ respectively the {\em affine Hilbert function}
and the {\em (generalized) Hilbert function of the finitely generated
right $F$-module $N$}. Observe that if $M$ (hence $N$) is a graded module,
that is, $M = \sum_{d\geq 0} M_d$ with $M_d = M\cap F^r_d$, then $\HF(N)(d) =
\dim_\KK F^r_d/M_d$. We denote by $\HS_a(N),\HS(N)$ the generating
series corresponding to $\HF_a(N),\HF(N)$, namely
\[
\HS_a(N) = \sum_{d\geq 0} \HF_a(N)(d) t^d\,,\,
\HS(N) = \sum_{d\geq 0} \HF(N)(d) t^d.
\]
We call such series respectively the {\em affine Hilbert series} and
the {\em (generalized) Hilbert series of $N$}. From the definition
it follows immediately that $\HS(N) = (1 - t)\HS_a(N)$.
\end{definition}

In literature one may find different names for the above functions
and series, especially in the case when $M\subset F$ is a two-sided
ideal and hence $N = F/M$ is an algebra. For instance, it is quite common
to call $\HF_a(N)$ the {\em growth function of $N$}. Sometimes one has
also the name of Henri Poincar\'e together with the one of David Hilbert
to denominate all these functions and series. It is well-known that
any Hilbert function can be obtained as the one which is defined by
a suitable monomial module. To this purpose we introduce the following
notions. 

\begin{definition}
Let $F^r$ be a free right module endowed with the canonical basis $\{e_i\}$.
We define $W(r) = \{e_i\}W = \{e_i w\mid 1\leq i\leq r, w\in W\}$ which is
a canonical linear basis of $F^r$. For this reason, the elements of $W(r)$ may be
called the {\em monomials of $F^r$}. Let now $\prec$ be a well-ordering of $W(r)$.
We call $\prec$ a {\em monomial ordering} of $F^r$ if $e_i w\prec e_j w'$
implies that $e_i w v\prec e_j w' v$, for all $e_i w, e_j w'\in W(r)$ and $v\in W$.
In particular, we say that $\prec$ is a {\em graded ordering} if $\deg(w) < \deg(w')$
implies that $e_i w\prec e_j w'$, for any $e_i w, e_j w'\in W(r)$.
\end{definition}

From now on, we assume that $F^r$ is endowed with a graded monomial ordering.
For instance, for all $e_i w, e_j w' \in W(r)$ one may define that $e_i w\prec e_j w'$
if and only if $w < w'$ in some graded lexicographic ordering or $w = w'$ and $i < j$.

\begin{definition}
Let $f = \sum_k e_{i_k} w_k c_k\in F^r$ with $e_{i_k} w_k\in W(r)$ and $c_k\in\KK,
c_k\neq 0$. If $e_{i_l} w_l = \max_\prec\{ e_{i_k} w_k \}$ then we put
$\lm(f) = e_{i_l} w_l$ and $\lc(f) = c_l$. Moreover, if $M\subset F^r$ is a right
submodule then one defines the monomial right submodule $\LM(M) =
\langle \lm(f) \mid f\in M, f\neq 0 \rangle\subset F^r$. We call $\LM(M)$
the {\em leading monomial module of $M$}. 
\end{definition}

The following is a generalization of a well-know theorem of Macaulay for
commutative modules (see, for instance, \cite{KR}).

\begin{theorem}
\label{macth}
Let $M\subset F^r$ be a right submodule and consider the right modules
$N = F^r/M$ and $N' = F^r/\LM(M)$. Then, one has that $\HF_a(N) = \HF_a(N')$
and hence $\HF(N) = \HF(N')$.
\end{theorem}

\begin{proof}
Put $M' = \LM(M)$ and denote $W(r)_{\leq d} = W(r)\cap F^r_{\leq d}$, for any
$d\geq 0$. Since $M'\subset F^r$ is a monomial submodule, one has clearly
that a linear basis of $F^r_{\leq d}/M'_{\leq d}$ is given by the set
$\{e_i w + M'_{\leq d}\mid e_i w\in W(r)_{\leq d}\setminus M'\}$.
Consider now any element $f\in F^r, f\neq 0$. If $\lm(f)\in M'$ then
there exists $g\in M$ such that $\lm(f) = \lm(g) v$, for some $v\in W$.
By putting $f_1 = f - g v \frac{\lc(f)}{\lc(g)}$ we obtain that $f \equiv f_1$ mod $M$
with $f_1 = 0$ or $\lm(f)\succ \lm(f_1)$. In the last case, we can repeat
this division step for the element $f_1\in F^r, f_1\neq 0$. Since $\prec$
is a well-ordering, we conclude that $f \equiv f_2$ mod $M$, for some $f_2\in F^r$
such that $f_2 = 0$ or $\lm(f)\succ \lm(f_2)\notin M'$. If $f_2\neq 0$
then we consider the element $f_3 = f_2 - \lm(f_2)\lc(f_2)$ and hence
$f - \lm(f_2)\lc(f_2)\equiv f_3$ mod $M$. Note that one has $f_3 = 0$
or $\lm(f_2)\succ \lm(f_3)$. By iterating the division process we finally
obtain that $f\equiv f'$ mod $M$ where $f' = \sum_k e_{i_k} w_k c_k$ with
$e_{i_k} w_k\in W(r)\setminus M'$ and $c_k\in \KK$. Moreover, one has clearly
that $f'\in M$ if and only if $f' = 0$. Recall now that $\prec$ is a graded
ordering and note that $\lm(f)\succ e_{i_k} w_k$ when $c_k\neq 0$.
We conclude that if $f\in F^r_{\leq d}$ then also $f'\in F^r_{\leq d}$.
\end{proof}

For the computation of the Hilbert series of a monomial module we are
reduced to the cyclic case. Precisely, it holds immediately the following result.

\begin{proposition}
\label{cycdec}
Let $M\subset F^r$ be a monomial right submodule, that is, $M$ is generated
by some set $\{e_{i_k} w_k\}\subset W(r)$. Then, one has that $M =
\bigoplus_i e_i I_i$ where $I_i\subset F$ is the monomial right ideal
which is generated by $\{w_k\mid e_{i_k} = e_i\}\subset W$. In particular,
if $N = F^r/M$ and $C_i = F/I_i$ then $\HS(N) = \sum_i \HS(C_i)$.
\end{proposition}

Let $\{w_k\}\subset W$ be a right basis of a monomial right ideal
$I\subset F$. Observe that $\{w_k\}$ is a minimal basis, that is,
for all $j\neq k$ there is no $v\in W$ such that $w_k = w_j v$ if and only if
$\{w_k\}$ is a free right basis. In other words, the ideal $I$ is a free
right submodule of $F$ according to the general property that
the algebra $F$ is a {\em free ideal ring} \cite{Co,Dr}.

\begin{proposition}
\label{fgform}
Let $\{w_k\}\subset W$ be a finite free basis of a (finitely generated)
monomial right ideal $I\subset F$ and define $C = F/I$ the corresponding
monomial cyclic right module. Then, one has that
\[
\HS(C) = \frac{1 - \sum_k t^{d_k}}{1 - nt}.
\]
\end{proposition}

\begin{proof}
Put $d_k = \deg(w_k)$ and denote by $F[-d_k]$ the algebra $F$ endowed
with the grading induced by the weight $d_k$, that is, we define
$F[-d_k]_d = 0$ if $d < d_k$ and $F[-d_k]_d = F_{d-d_k}$ otherwise.
In a similar way, one defines the graded free right $F$-module
$\bigoplus_k F[-d_k]$. If $\{e_k\}$ is the canonical basis of such
module then by definition $\deg(e_k) = \delta_k$, for all $k$.
We consider now the graded right $F$-module homomorphism
\[
\bigoplus_k F[-d_k]\to F,\ \sum_k e_k f_k\mapsto \sum_k w_k f_k
\]
which implies the following short exact sequence
\[
0\to \bigoplus_k F[-d_k]\to F\to C\to 0.
\]
Since the Hilbert series of $F[-d_k]$ is clearly $t^{d_k}/(1 - n t)$
we obtain that
\[
\HS(C) = \frac{1}{1 - nt} - \sum_k \frac{t^{d_k}}{1 - n t} =
\frac{1 - \sum_k t^{d_k}}{1 - nt}.
\]
\end{proof}

Note that by Proposition \ref{cycdec} one obtains also the rationality of the
Hilbert series of a finitely presented (monomial) right module.
Unfortunately, the free associative algebra $F = \KK\langle X \rangle$ is not
a right Noetherian ring. Consider, for instance, the right ideal
$I = \langle x_1, x_2x_1, x_2^2x_1, \ldots \rangle$. Then, Proposition \ref{fgform}
does not generally apply and we have to find a different approach to
the computation of the Hilbert series of infinitely presented monomial
cyclic right modules.

\section{A key short exact sequence}

Let $I\subset F$ be any monomial right ideal and consider the corresponding
monomial cyclic right module $C = F/I$. If $F[-1]$ is the algebra $F$
endowed with the grading induced by the weight 1 then we define $I[-1]$
as the right ideal $I$ endowed with the grading of $F[-1]$ and we put
$C[-1] = F[-1]/I[-1]$. Moreover, we denote by $F[-1]^n$ the graded free
right $F$-module which is the direct sum of $n$ copies of $F[-1]$.
Denote by $\{e_i\}$ the canonical basis of $F[-1]^n$ and therefore
$\deg(e_i) = 1$, for all $1\leq i\leq n$. We consider the graded right
$F$-module homomorphism
\[
\varphi:F[-1]^n\to C,\ \sum_i e_i f_i\mapsto \sum_i x_i f_i.
\]
The image $\Im\varphi$ is clearly the graded right submodule
$B = \langle x_1,\ldots,x_n \rangle\subset C$. One has immediately that
the cokernel $C/B$ is either 0 when $C = 0$ or it is isomorphic to the base
field $\KK$ otherwise.

We need to describe the kernel $\Ker\varphi$ which is a graded right
submodule of $F[-1]^n$. For this purpose we make use of the following notion
(see, for instance, \cite{Mo}). For each element $f\in F$, one defines
the {\em colon right ideal}
\[
(I :_R f) = \{g\in F\mid f g\in I\}.
\]
This set is in fact a right ideal of $F$ which generally does not contain $I$.
Nevertheless, one has that $I\subset (I :_R f)$ when $I$ is a two-sided ideal.
It is clear that $(I :_R f) = \langle 1 \rangle$ if and only if $f\in I$.
Because $I$ is a monomial ideal, we have that $(I :_R w)$ is also a monomial ideal,
for any $w\in W$. In Proposition \ref{rprop} and Proposition \ref{tsprop}
we will provide methods to compute a monomial basis of $(I :_R w)$ starting
from a monomial basis of $I$.

Since we are in the monomial case, one has immediately that $\ker\varphi$
is isomorphic to the graded right submodule $\bigoplus_i (I :_R x_i)\subset F[-1]^n$.
To simplify notations, we put $I_{x_i} = (I :_R x_i)$ and we denote by $C_{x_i} =
F/I_{x_i}$ the corresponding monomial cyclic right module. One obtains the
following short exact sequence of graded right $F$-module homomorphisms
\begin{equation}
\label{exseq}
0 \to \bigoplus_{1\leq i\leq n} C_{x_i}[-1]\to C\to C/B\to 0.
\end{equation}
It is useful to introduce the following number
\[
\gamma(I) =
\left\{
\begin{array}{cl}
1 & \mbox{if}\ I = \langle 1 \rangle, \\
0 & \mbox{otherwise.}
\end{array}
\right.
\]
In other words, we have that $\dim_\KK (C/B) = 1 - \gamma(I)$. By the exact
sequence (\ref{exseq}) one obtains the following key formula for the
corresponding Hilbert series
\begin{equation}
\label{form}
\HS(C) = 1 - \gamma(I) + t \sum_{1\leq i\leq n} \HS(C_{x_i}).
\end{equation}

The above formula suggests that one may compute a rational form for the sum of
the Hilbert series $\HS(C)$ by successively transforming a monomial right ideal $I$
into its colon right ideals $(I :_R x_i)$, for all variables $x_i\in X$. We analize
this method in the next section under the assumption that there are only a finite
number of transforms to be performed. When such condition is satisfied will be
characterized in Section 5.

\section{Computing the Hilbert series}

It is useful to think the colon right ideal operation as a mapping on the set
of (monomial) right ideals of $F$.

\begin{definition}
Denote by $\mi$ the set of all monomial right ideals of $F$. For all $1\leq i\leq n$,
we define the function $T_{x_i}:\mi\to\mi$ such that $T_{x_i}(I) = (I :_R x_i)$,
for any monomial right ideal $I\in\mi$.
\end{definition}

\begin{definition}
Let $I\in\mi$. Define $\O_I\subset\mi$ the minimal subset containing $I$
such that $T_{x_i}(\O_I)\subset \O_I$, for all $1\leq i\leq n$. We call $\O_I$
the {\em orbit of $I$}. One has clearly that $\O_J\subset\O_I$, for any $J\in\O_I$.
\end{definition}

\begin{definition}
We call $I\in\mi$ a {\em regular (monomial) right ideal} if its orbit $\O_I$
is a finite set and we denote by $\ri$ the set of all such ideals.
Since $T_{x_i}(I)\in\O_I$ observe that $T_{x_i}$ maps the set $\ri$ into itself,
for each $1\leq i\leq n$.
\end{definition}

\begin{definition}
Let $I\in\ri$ with $\O_I = \{I_1,\ldots,I_r\}$. We define $A_I = (a_{kl})$
the integer $r$-by-$r$ matrix such that
\[
a_{kl} = \# \{1\leq i\leq n\mid T_{x_i}(I_k) = I_l\}.
\]
Moreover, we denote by $p_I(t)\in\ZZ[t]$ the characteristic polynomial of $A_I$
and we consider the 0-1 column $r$-vector
\[
\C_I = (1 - \gamma(I_1),\ldots,1 - \gamma(I_r))^T.
\]
We call $A_I, p_I(t)$ and $\C_I$ respectively the {\em adjacency matrix,
characteristic polynomial and constant vector of (the orbit of) $I$}.
\end{definition}

In the above definition one has clearly to consider $\O_I$ as an ordered set.
Note also that in the vector $\C_I$ there is a single entry equal to 0
corresponding to the trivial ideal $F = \langle 1 \rangle$. In fact, such ideal
has to belong to the orbit $\O_I$ since for any monomial
$w = x_{i_1}\cdots x_{i_d}\in I$ one has that
\[
(T_{x_{i_d}}\cdots T_{x_{i_1}})(I) =
(((I :_R x_{i_1}) \cdots ) :_R x_{i_d}) =
(I :_R w) = \langle 1 \rangle.
\]

\begin{theorem}
\label{mainth}
Let $I\in\ri$ and define $C = F/I$ the corresponding monomial cyclic right module.
Then, the Hilbert series $\HS(C)$ is a rational function with integer
coefficients.
\end{theorem}

\begin{proof}
Let $\O_I = \{I_1,\ldots,I_r\}$ and put $C_k = F/I_k$ ($1\leq k\leq r)$.
Moreover, we put $H_k = \HS(C_k)$ and $c_k = 1 - \gamma(I_k)$. By applying
the formula (\ref{form}) for each cyclic module $C_k$, one obtains $r$ linear
equations in the $r$ unknowns $H_k$, namely
\[
H_k - t \sum_{1\leq i\leq n} H_{l_{ki}} = c_k\ (1\leq k\leq r)
\]
where the index $1\leq l_{ki}\leq r$ is defined such that $I_{l_{ki}} = T_{x_i}(I_k)$.
Note that these equations are with coefficients in the rational function field
$\QQ(t)$. By definition of the adjacency matrix $A_I$ and the constant vector $\C_I$,
one has therefore that the column vector $\vH = (H_1,\ldots,H_r)^T$
is a solution of the matrix equation
\begin{equation}
\label{mateq}
(I - t A_I) \vH = \C_I.
\end{equation}
Observe now that $\det(I - t A_I) = t^r p_I(1/t)\neq 0$ since $p_I(t) =
\det(t I - A_I)$ is the characteristic polynomial of the adjacency matrix
$A_I$. We conclude that the equation (\ref{mateq}) has a unique solution
$\vH = (I - t A_I)^{-1} \C_I$.
\end{proof}

An immediate consequence of the above result is the following one.

\begin{corollary}
Let $M\subset F^r$ be a monomial right submodule and denote by
$N = F^r/M$ the corresponding finitely generated monomial right module.
Moreover, denote by $M = \bigoplus_i e_i I_i$ the decomposition of $M$ into
monomial right ideals according to Proposition \ref{cycdec}.
If all the ideals $I_i$ are regular ones then the Hilbert series $\HS(N)$
is a rational function with integer coefficients.
\end{corollary}

The numerator and the denominator of the Hilbert series of a regular right
ideal can be explicitly obtained in the following way.

\begin{definition}
Let $I\in\ri$ and assume $\O_I = \{I_1,\ldots,I_r\}$ with
$I_1 = I$. Denote by $\B_I$ the column $r$-vector which coincides with
the first column of the cofactor matrix of the $r$-by-$r$ matrix $I - t A_I$.
Note that the entries of $\B_I$ are polynomials with integer coefficients.
We call $\B_I$ the {\em polynomial vector of (the orbit) of $I$}.
\end{definition}

The matrix equation (\ref{mateq}) implies immediately the following result.

\begin{corollary}
Let $I\in\ri$ and consider $C = F/I$. The rational function $\HS(C) = f(t)/g(t)$
$(f(t),g(t)\in\ZZ[t])$ is such that
\begin{equation}
\label{form2}
f(t) = \B_I^T \C_I, g(t) = t^r p_I(1/t)
\end{equation}
where $r = \# \O_I = \deg(p_I(t))$.
\end{corollary}

Observe in the above result that the entry of $\B_I$ corresponding to the
zero entry of $\C_I$ ($\langle 1\rangle\in\O_I$) does not clearly contribute
to the formation of the numerator $f(t)$.
We present now a simple algorithm to compute the data defining the matrix
equation (\ref{mateq}) corresponding to a regular right ideal.

\newpage
\suppressfloats[b]
\floatname{algorithm}{Algorithm}
\begin{algorithm}\caption{\OrbitData}
\begin{algorithmic}[0]
\State \text{Input:} $I\in\ri$.
\State \text{Output:} $(\O_I,A_I,\C_I)$.
\State $\N:= \{I\}$;
\State $\O:= \{I\}$;
\State $A = (a_{kl}) := 0$ (square matrix);
\State $\C = (c_k)^T := 0$ (column vector);
\While{$\N\neq\emptyset$}
\State {{\bf choose}} $J\in\N$;
\State $\N:= \N\setminus\{J\}$;
\State $k:= $ the position of $J$ in $\O$;
\If{$J\neq \langle 1\rangle$}
\State $c_k:= 1$;
\EndIf;
\ForAll{$1\leq i\leq n$}
\State $J':= T_{x_i}(J)$;
\If{$J'\notin\O$}
\State $\N:= \N\cup\{J'\}$;
\State $\O:= \O\cup\{J'\}$;
\EndIf;
\State $l:= $ the position of $J'$ in $\O$;
\State $a_{kl}:= a_{kl} + 1$;
\EndFor;
\EndWhile;
\State \Return $(\O,A,\C)$.
\end{algorithmic}
\end{algorithm}

Note explicitly that $A$ and $\C$ are in fact a square matrix and a vector
of increasing dimension. When the algorithm stops, such dimension is exactly
the cardinality of $\O$. Then, the Hilbert series corresponding to the regular
right ideal $I$ is obtained by using the formula (\ref{form2})
or equivalently by solving the matrix equation $(I - t A)\vH = \C$.
Observe that the matrix $(I - t A)$ results sparse if the number of ideals
in the orbit $\O$ is large with respect to the number of variables $x_i$.

We describe now how the mapping $T_{x_i}$ is actually defined in terms of
a monomial basis of $I$. 

\begin{proposition}
\label{rprop}
Let $\{w_j\}\subset W$ be a right basis of a monomial right ideal $I\subset F$
and consider a monomial $w\in W$. For all $j$, we define the element
\begin{equation}
w'_j =
\left\{
\begin{array}{cl}
 1  & \mbox{if}\ w = w_j v_j\ (v_j\in W), \\
v_j & \mbox{if}\ w_j = w v_j, \\
 0  & \mbox{otherwise}.
\end{array}
\right.
\end{equation}
Then, a right basis of $(I :_R w)$ is given by the set $\{w'_j\}\subset W$.
In particular, if $w = x_i$, that is, $(I :_R w) = T_{x_i}(I)$ then
the right basis $\{w'_j\}$ is defined according to the following rule
\begin{equation}
w'_j =
\left\{
\begin{array}{cl}
 1  & \mbox{if}\ w_j = 1\ \mbox{or}\ w_j = x_i, \\
v_j & \mbox{if}\ w_j = x_i v_j\ (v_j\in W), \\
 0  & \mbox{otherwise}.
\end{array}
\right.
\end{equation}
\end{proposition}

\begin{proof}
If a monomial $t\in W$ belongs to $(I :_R w)$ then by definition $w t\in I$
and therefore $w t = w_j t'$, for some $t'\in W$ and some index $j$. One has two cases.
If $w = w_j v_j$ (hence $v_j t = t'$) for some $v_j\in W$, then $w\in I$ and therefore
$(I :_R w) = \langle 1 \rangle$. Otherwise, there is $v_j\in W$ such that $w_j = w v_j$
and $t = v_j t'$.
\end{proof}

An important case is when $I$ is a monomial two-sided ideal, that is, $C = F/I$
is a monomial algebra. Recall that a two-sided ideal is always contained in
the colon right ideals it defines.

\begin{proposition}
\label{tsprop}
Let $\{w_j\}\subset W$ be a two-sided basis of a monomial two-sided ideal
$I\subset F$ and consider $w\in W$. Assume that $w\notin I$, that is,
$(I :_R w)\neq \langle 1 \rangle$. For all $j$, we define the (finitely
generated) monomial right ideal
\[
I_w(w_j) = \langle v_{jk}\mid u_{jk} w_j = w v_{jk}, u_{jk},v_{jk}\in W,
\deg(v_{jk}) < \deg(w_j) \rangle.
\]
Then, one has that $(I :_R w) = \sum_j I_w(w_j) + I$.
\end{proposition}

\begin{proof}
If $\{w_j\}$ is a monomial two-sided basis of $I$ then we have that the set
$W\{w_j\} = \{ u w_j\mid u\in W, j\,\mbox{any}\}$ is a monomial
right basis of $I$. By Proposition \ref{rprop} one has therefore that a right basis
of $(I :_R w)$ is given by the monomials $v_{jk}\in W$ such that
$u_{jk} w_j = w v_{jk}$, for some $u_{jk}\in W$. Moreover,
if $\deg(v_{jk})\geq \deg(w_j)$ then we have clearly that $v_{jk}\in I$.
\end{proof}

From the above results it is clear that if a monomial ideal $I\subset F$ is finitely
generated in the right or two-sided case then the mappings $T_{x_i}$ can be
effectively performed and the orbit $\O_I$ is finite. Observe explicitly that
if $I$ is a finitely generated two-sided ideal then $I$ is generally infinitely
generated as a right ideal and Proposition \ref{fgform} does not apply.
In this case, we can apply instead $\OrbitData$ which is an effective procedure
that can be easily implemented in any computer algebra system. We will provide
in fact some implementation and application of this algorithm in Section 7.
Note that commutative ideals are always finitely generated and hence $\OrbitData$
can be used also for computing commutative Hilbert series. We remark finally that
the rationality of the Hilbert series of a finitely presented monomial algebra
was first proved by Govorov \cite{Go} and further investigated by Anick \cite{An}
and Ufnarovski \cite{Uf1,Uf2,Uf3}.

Owing to non-Noetherianity of the free associative algebra, the monomial ideal $I$
may be infinitely generated also in the two-sided case and hence one can object
about the possibility to compute the transforms $T_{x_i}$. As a matter of fact,
we will show in the next section that the case when $I$ is regular, that is,
the orbit $\O_I$ is finite is also the case when the monomials of $I$ can be
represented by finitely many ``regular expressions'' which make possible
the computation of the $T_{x_i}$. We will apply this approach to a concrete
(infinitely generated) regular two-sided ideal in Section 6.

\section{Automata of monomial right ideals}

To the aim of understanding which monomial right ideals are regular, it is
convenient to introduce here the concept of deterministic automaton which is
a fundamental one in information theory and engineering. Recall that we have
denoted by $W$ the monoid of all monomials of $F = \KK\langle X\rangle$, that is,
the elements of $W$ are words over the alphabet $X = \{x_1,\ldots,x_n\}$.
Another usual notation for this monoid is $X^*$. Moreover, we denote by $W^{op}$
the opposite monoid of $W$.

\begin{definition}
Let $Q$ be any set. Fix an element $q_0\in Q$ and a subset $A\subset Q$.
By denoting $\F(Q)$ the monoid of all functions $Q\to Q$, we consider also a monoid
homomorphism $T:W^{op}\to \F(Q)$. In other words, the mapping $T$ defines a monoid
right action of $W$ on the set $Q$. By definition, a {\em deterministic automaton}
is any 5-tuple $(Q,X,T,q_0,A)$. This is called a {\em deterministic finite automaton},
briefly a {\em DFA}, when $Q$ is a finite set. The elements of $Q$ are usually called
the {\em states of the automaton} where $q_0$ is the {\em initial state}. A state
$q\in A$ is called an {\em accepting state}. The homomorphism $T$ is called
the {\em transition function of the automaton}. For this map we make use of the
notation $w = x_{i_1}\cdots x_{i_d}\mapsto T_w = T_{x_{i_d}}\cdots T_{x_{i_1}}$.
\end{definition}

Strictly related to the notion of automaton are the following concepts.

\begin{definition}
Any subset $L\subset W$ is called a {\em (formal) language}. A deterministic automaton
$(Q,X,T,q_0,A)$ {\em recognizes} the language $L$ if $L = \{w\in W\mid T_w(q_0)\in A\}$.
A language is called {\em regular} if it is recognized by a FDA. A FDA which recognizes
a regular language $L$ is called {\em minimal} if the number of its states
is minimal with respect of any other FDA recognizing $L$.
\end{definition}

Clearly, if the automaton $(Q,X,T,q_0,A)$ recognizes $L$ then $(Q,X,T,q_0,A^c)$
recognizes $L^c$ where $A^c,L^c$ are the complements of $A,L$ in $Q,W$, respectively.
It follows immediately that $L$ is a regular language if and only if $L^c$ is such.
For any language $L\subset W$, observe that the monoid structure of $W$ canonically
defines an infinite automaton recognizing $L$. Precisely, such automaton is
$(W,X,T,1,L)$ where $T_v(w) = w v$, for all $v,w\in W$.

\begin{definition}
Let $L\subset W$ and define the set $w^{-1}L = \{v\in W\mid w v\in L\}$,
for all $w\in W$. The {\em Nerode equivalence defined by $L$} is by definition
the following equivalence relation on $W$
\[
w \sim w'\ (w,w'\in W)\ \mbox{if and only if}\ w^{-1}L = w'^{-1}L.
\]
Observe that the quotient set $W/\sim$ is in one-to-one correspondence with
the set $\{w^{-1}L\mid w\in W\}$. Moreover, we have that the relation $\sim$
is right invariant, that is, $w\sim w'$ implies that $w v\sim w' v$, for all
$v,w,w'\in W$. Then, an induced automaton $(W/\sim,X,\bar{T},[1],[L])$
recognizing $L$ is defined by putting $\bar{T}_{v}([w]) = [w v]$, for any $v,w\in W$
and $[L] = \{[w]\mid w\in L\}$. This is called the {\em Nerode automaton
recognizing $L$}.
\end{definition}

Fundamental results in automata theory are the following ones (see, for instance,
\cite{DLV}).

\begin{theorem}[Myhill-Nerode]
A language $L\subset W$ is regular if and only if the quotient set $W/\sim$
is finite. In this case, the Nerode automaton is a minimal FDA recognizing $L$.
\end{theorem}

Given two languages $L,L'\subset W$, one can define their set-theoretic union
$L\cup L'$ and their product $L L' = \{w w'\mid w\in L,w'\in L'\}$. For any $d\geq 0$,
we have also the power $L^d = \{w_1\cdots w_d\mid w_i\in L\}$ ($L^0 = \{1\}$) and
the {\em star operation} $L^* = \bigcup_{d\geq 0} L^d$. In other words,
the language $L^*$ is the (free) submonoid of $W$ which is generated by $L$. 
The union, the product and the star operation are called the {\em rational operations}
over the languages.

\begin{theorem}[Kleene]
A language $L\subset W$ is regular if and only if it can be obtained from
finite languages by applying a finite number of rational operations.
\end{theorem}

\begin{example}
The language $L = \{(x_1x_2^i)^j\mid i,j\geq 0\}\cup \{(x_2x_1)^k\mid k\geq 0\}$
is regular since it is obtained from the finite languages $L_1 = \{x_1\},
L_2 = \{x_2\}$ by rational operations, namely $L = (L_1 L_2^*)^*\cup (L_2L_1)^*$.
\end{example}

We show now how automata theory applies to our approach to compute noncommutative
Hilbert series. Recall that $\mi$ denotes the set of all monomial right
ideals of $F$ and for each variable $x_i\in X$ we have map $T_{x_i}:\mi\to\mi$
such that $T_{x_i}(I) = (I :_R x_i)$, for any $I\in\mi$. We have therefore
a monoid homomorphism $T:W^{op}\to\F(\mi)$ such that $w = x_{i_1}\ldots x_{i_d}\mapsto
T_w = T_{x_{i_d}}\cdots T_{x_{i_1}}$. For each monomial right ideal $I\in\mi$,
we have clearly that
\[
T_w(I) = (I :_R w) = \{ f\in F\mid w f\in I\}.
\]
Fix now $I\in\mi$. One has immediately that $\O_I = \{T_w(I)\mid w\in W\}$
and hence we can restrict the right action of $W$ that $T$ defines on the
set $\mi$ to an action on the orbit $\O_I$. By abuse of notation,
we will denote the corresponding monoid homomorphism as $T:W^{op}\to \F(\O_I)$.
By considering the singleton $\{\langle 1\rangle\}\subset\O_I$, one has
then a deterministic automaton $(\O_I,X,T,I,\{\langle 1\rangle\})$.
The language $L$ which is recognized by this automaton is by definition
the set of monomials $w\in W$ such that $T_w(I) = (I :_R w) = \langle 1 \rangle$,
that is, $w\in I$. In words, we have that $L = I\cap W$. Observe that
the complementary automaton $(\O_I,X,T,I,\O_I\setminus\{\langle 1\rangle\})$
recognizes exactly the language $L = W\setminus I$ of the normal monomials
modulo $I$.

Recall now that the states of the Nerode automaton recognizing $L$ are in
one-to-one correspondence with the sets $w^{-1} L = \{v\in W\mid w v\in L\}$,
for all $w\in W$. Since $L = I\cap W$, it is clear now that $w^{-1} L =
(I :_R w)\cap W$ and hence one obtains the following result.

\begin{proposition}
\label{minimal}
The monomial right ideal $I\subset F$ is regular if and only if the corresponding
language $L = I\cap W$ is regular. In this case, the algorithm $\OrbitData$ computes
a FDA $(\O_I,X,T,I,\{\langle 1\rangle\})$ recognizing $L$ which is a minimal one.
\end{proposition}

By the above proposition it follows that, in the regular case, the algorithm
$\OrbitData$ is an optimal one for computing the rational Hilbert series of
a monomial cyclic right module $C = F/I$ by means of the exact sequence (\ref{exseq}).
Moreover, we have the Kleene's theorem to verify if the monomial right ideal $I$
is a regular one.

\section{An illustrative example}

For illustrating how the algorithm $\OrbitData$ works in practice with infinitely
generated but regular monomial ideals, we consider an algebra which is associated
to the presentation of an Artin group defined by the following Coxeter matrix
\[
C = 
\left(
\begin{array}{ccc}
- & \infty & 3 \\
\infty & - & 2 \\
3 & 2 & -
\end{array}
\right).
\]
In other words, we consider the algebra $A = F/J$ where
$F = \KK\langle x,y,z \rangle$ and $J\subset F$ is the two-sided ideal
generated by the two binomials $y z - z y, x z x - z x z$. The base field
$\KK$ may be any. To the aim of computing the Hilbert series of $A$, we calculate
a two-sided \Gr\ basis $G$ of $J$ with respect to the graded left
lexicographic monomial ordering of $F$ with $x\succ y\succ z$.
It is immediate to verify that $G$ is the following infinite set
\[
G = \{yz - zy, xzx - zxz\}\cup \{x z^2 z^d x z - z x z^2 x x^d\mid d\geq 0\}.
\]
This implies that the leading monomial two-sided ideal $I = \LM(J)$ is as follows
\[
I = \langle yz, xzx \rangle + \langle x z^2 z^d x z\mid d\geq 0 \rangle.
\]
If $X = \{x,y,z\}$ and $W = X^*$ then the language $L = I\cap W$ of the
monomials of $I$ can be clearly written as
\[
L = X^* \{yz, xzx\} X^* \cup X^* \{xz^2\}\{z\}^*\{xz\} X^*.
\]
By the Kleene's theorem we have that $L$ is a regular language and hence
the orbit $\O_I$ is a finite set and the Hilbert series $\HS(A)$ is a rational
function. Recall that to calculate $\O_I$ we have to compute the colon right ideals
$I_w = T_w(I) = (I :_R w)$ ($w\in W$). Moreover, by Proposition \ref{tsprop}
one has a method to compute the colon right ideals that are defined by
a monomial two-sided ideal. Precisely, for computing $I_x = (I :_R x)$
we have to consider the following monomial right ideals
\[
I_x(yz) = 0, I_x(xzx) = \langle zx \rangle,
I_x(x z^2 z^d x z) = \langle z^2 z^d x z \rangle\ (d\geq 0).
\]
One obtains therefore that
\[
I_x = \langle zx \rangle + \langle z^2 z^d x z \mid d\geq 0 \rangle + I.
\]
In a similar way, we compute that $I_y = \langle z \rangle + I$
and we have that $\{I, I_x, I_y\}\subset\O_I$. Moreover, one has immediately
that $I_z = I$. We compute now
\[
I_{x^2}(yz) = 0, I_{x^2}(xzx) = \langle zx \rangle,
I_{x^2}(x z^2 z^d x z) = \langle z^2 z^d x z \rangle\ (d\geq 0)
\]
and one concludes that $I_{x^2} = I_x$. By considering the ideals
\[
I_{xy}(yz) = \langle z \rangle, I_{xy}(xzx) = 0,
I_{xy}(x z^2 z^d x z) = 0\ (d\geq 0)
\]
we obtain also that $I_{xy} = I_y$. In a similar way, one has that
\[
I_{xz} = \langle x \rangle + \langle z z^d x z \mid d\geq 0 \rangle + I
\]
and hence $\{I, I_x, I_y, I_{xz}\}\subset\O_I$. It is immediate to calculate
now $I_{yx} =  I_x, I_{y^2} = I_y$ and $I_{yz} = \langle 1 \rangle$ because
$yz\in I$. One obtains then $\{I, I_x, I_y, I_{xz}, I_{yz}\}\subset \O_I$.
Clearly $I_{xzx} = \langle 1 \rangle$ and we have that $I_{xzy} = I_y$.
One computes also that
\[
I_{xz^2} = \langle z^d x z \mid d\geq 0 \rangle + I
\]
and therefore $\{I, I_x, I_y, I_{xz}, I_{yz}, I_{xz^2}\}\subset\O_I$.
It is clear that the ideal $I_{yz} = \langle 1 \rangle$ is invariant under
the maps $T_x,T_y,T_z$ and hence we apply them to the ideal $I_{xz^2}$.
One computes the following monomial right ideals
\begin{equation*}
\begin{gathered}
I_{xz^2x}(yz) = 0, I_{xz^2x}(xzx) = \langle zx \rangle, \\
I_{xz^2x}(x z^2 x z) = \langle z \rangle,
I_{xz^2x}(x z^2 z^d x z) = \langle z^2 z^d x z \rangle\ (d > 0)
\end{gathered}
\end{equation*}
and we have therefore that $I_{xz^2x} = I_y$. In a similar way, one has that
$I_{xz^2y} = I_y$ and $I_{xz^3} = I_{xz^2}$ and we conclude that
\[
\O_I = \{I, I_x, I_y, I_{xz}, I_{yz}, I_{xz^2}\}.
\]
Then, for the regular ideal $I$ one has the following adjacency matrix
\[
A_I =
\left(
\begin{array}{cccccc}
1 & 1 & 1 & 0 & 0 & 0 \\
0 & 1 & 1 & 1 & 0 & 0 \\
0 & 1 & 1 & 0 & 1 & 0 \\
0 & 0 & 1 & 0 & 1 & 1 \\
0 & 0 & 0 & 0 & 3 & 0 \\
0 & 0 & 2 & 0 & 0 & 1 \\
\end{array}
\right)
\]
and the constant vector $\C_I = (1,1,1,1,1,0,1)^T$. Denote now by $\vH$
the column vector whose entries are the sums of the Hilbert series
of the cyclic right modules defined by the right ideals in the orbit $\O_I$.
From Theorem \ref{mainth} it follows that $\vH$ is obtained by solving over
the field $\QQ(t)$ the linear system corresponding to the matrix equation
(\ref{mateq}). In particular, we obtain easily that
\[
\HS(A) = \frac{1}{(t - 1)(t^2 + 2t - 1)}.
\]
Since the roots of the denominator are $1, -1\pm \sqrt{2}$, note finally
that the algebra $A = F/J$ has an exponential growth with exponential dimension
(see, for instance, \cite{Uf3})
\[
\limsup_{d\to\infty} \root d \of {\dim_\KK A_d} = \frac{1}{-1 + \sqrt{2}}.
\]

\section{Experiments}

By means of an experimental implementation of \OrbitData\ that we have
developed using the language of Maple, we propose now the computation
of the (generalized) Hilbert series of a couple of Hecke algebras which have
a finite \Gr\ basis for the two-sided ideal of the relations satisfied
by their generators. Precisely, consider the following two Coxeter matrices
of order 4
\[
C =
\left(
\begin{array}{
@{\hskip 2pt}c@{\hskip 2pt}c@{\hskip 2pt}c@{\hskip 2pt}c@{\hskip 2pt}
}
1 & 3 & 2 & 3 \\
3 & 1 & 3 & 2 \\
2 & 3 & 1 & 3 \\
3 & 2 & 3 & 1 \\
\end{array}
\right)\,,\,
C' =
\left(
\begin{array}{
@{\hskip 2pt}c@{\hskip 2pt}c@{\hskip 2pt}c@{\hskip 2pt}c@{\hskip 2pt}
}
1 & 3 & 3 & 2 \\
3 & 1 & 3 & 3 \\
3 & 3 & 1 & 3 \\
2 & 3 & 3 & 1 \\
\end{array}
\right).
\]
Let $F = \KK\langle x,y,z,v \rangle$ be the free associative algebra
with four generators and define the following two-sided ideals of $F$
\begin{equation*}
\begin{gathered}
J = \langle
(x  - q)(x + 1),
(y  - q)(y + 1),
(z  - q)(z + 1),
(v  - q)(v + 1), \\
x y x - y x y,
x z - z x,
x v x - v x v,
y z y - z y z,
y v - v y,
z v z - v z v
\rangle;
\\
J' = \langle
(x  - q)(x + 1),
(y  - q)(y + 1),
(z  - q)(z + 1),
(v  - q)(v + 1), \\
x y x - y x y,
x z x - z x z,
x v - v x,
y z y - z y z,
y v y - v y v,
z v z - v z v
\rangle.
\end{gathered}
\end{equation*}
Then, the Hecke algebras corresponding to the matrices $C,C'$ are
by definition the quotient algebras $A = F/J, A' = F/J'$. Observe that
the coefficients of the generators of $J,J'$ are $\pm 1$ together with
the parameter ``$q$''. In fact, one has the same situation for the \Gr\ bases
of these ideals, that is, the base field $\KK$ may be any field containing $q$.
For the graded left lexicographic monomial ordering of $F$
with $x\succ y\succ z\succ v$, one computes (see for instance \cite{LSL,LS})
the following finitely generated leading monomial ideals
\begin{equation*}
\begin{gathered}
\LM(J) = \langle
x^2, x z, y^2, y v, z^2, v^2, x y x, x v x, y z y, z v z, x y z x,
x v z x, x v z v, y z v y, \\
x v y x y, x v y x v, x y z v x v, x v z y x y, x v y z x y z,
x v y z x v z, x v z y x v y, x v y z v x v z, \\
x v z y z x y z, x v y z x v y z v, x v z y x v z y z,
x v z y z x v y z v, x v z y z x v z y z v
\rangle;
\\
\LM(J') = \langle
x^2, x v, y^2, z^2, v^2, x y x, x z x, y z y, y v y, z v z, x y v x,
x z v x, x y z x z, \\
x z y x y, y z v y v, y v z y z, x z v y x y, x y v z x z, x z y x z y z,
y v z y v z v, x z y v x y v, \\
x y z v x z v, x z v y x z y z, x z v y v x y v, x y v z v x z v, x z y x z y v z v,
x z v y x z y v z v
\rangle.
\end{gathered}
\end{equation*}
By performing \OrbitData\ for the monomial two-sided ideals $I = \LM(J)$ and
$I' = \LM(J')$, we obtain the following Hilbert series
\[
\HS(A) = \frac{(1 + t)(1 + t^2)}{(1 - t)^3}\,,\,
\HS(A') = \frac{(1 + t)(1 + t^2)(1 + t + t^2)}{(1 - t)(1 - t - t^2 - t^3 - t^4)}.
\]
From the roots of the denominators of such series it follows that the algebra $A$
has a polynomial growth but the growth of $A'$ is an exponential one.
The number of elements in the orbits $\O_I,\O_{I'}$ are respectively 36 and 33.
This implies that the linear system corresponding to the matrix equation
(\ref{mateq}) is quite sparse and hence easy to solve for both the orbits.
In other words, the solving time is irrelevant with respect to the computation
of the orbit. With our experimental implementation in the language of Maple,
the total computing times for $A,A'$ are respectively 0.9 and 0.7 sec. We obtain
such timings on a server running Maple 16 with a four core Intel Xeon at 3.16GHz
and 64 GB RAM.

\section{Conclusions and future directions}

The above experiments clearly show that the proposed method for computing
noncommutative Hilbert series is really feasible. Beside the ease of
implementing it by using Proposition \ref{rprop} and \ref{tsprop}, we believe
that this good behaviour relies on the property that the deterministic finite
automaton which is constructed by the algorithm \OrbitData\ is a minimal one
according to Proposition \ref{minimal}. Since the monomial exact sequence that
we use in Section 3 is similar to the ones which are generally used for the
computation of commutative Hilbert series (see, for instance, \cite{KR}),
one possible further investigation may consist in applying our method to
the commutative case and in comparing it with other existing strategies.
Of course, this can be done properly only when \OrbitData\ will be implemented
in the kernel of a computer algebra system. Such implementation will make
also available the Hilbert series of involved noncommutative algebras
and modules which are of interest in mathematics and physics.

One drawback of these monomial methods consists in the necessity of computing
\Gr\ bases in order to obtain leading monomial ideals. In the noncommutative case
these bases may be infinite but in fact we have proved in Section 5 and illustrated
in Section 6 that our method works also in the infinitely generated but regular
monomial case. Nevertheless, one may have, for instance, a finite minimal
presentation of a graded algebra and hence we may prefer to use this finite data
to try to determine its Hilbert series. This eventually leads to the computation
of the homology of the algebra which is of course of great interest in itself.
We hope therefore in the next future to provide new methods also for such
computations.

\section{Acknowledgments}

The author would like to thank Victor Ufnarovski for introducing him
to automata theory. He is also grateful to the research group of {\sc Singular}
\cite{DGPS} for the courtesy of allowing access to their servers for performing
computational experiments. Finally, the author would thank the anonymous
referees for the opportunity to apply the Occam's razor to the paper
once more.

%%%%%%%%%%%%%%%%%%%%%%%%%%%%%%%%%%%%%%%%%%%%%%%%%%%%%%%%%%%%%%%%%%%%%%

\end{document}